\documentclass[12pt]{article}
\usepackage{amsmath,amssymb,amsthm, amsfonts}
\usepackage{hyperref}
\usepackage{graphicx}
\usepackage{url}
\usepackage[mathlines]{lineno}
\usepackage{dsfont} 
\usepackage{xcolor}
\definecolor{red}{rgb}{1,0,0}

\definecolor{blue}{rgb}{0,0,1}

\definecolor{green}{rgb}{0,.6,0}

\newtheorem{thm}{Theorem}[section]
\newtheorem{cor}[thm]{Corollary}

\newtheorem{prop}[thm]{Proposition}
\newtheorem{conj}[thm]{Conjecture}
\newtheorem{obs}[thm]{Observation}

\theoremstyle{definition}

\theoremstyle{definition}
\newtheorem{defn}[thm]{Definition}

\theoremstyle{definition}
\newtheorem{ex}[thm]{Example}

\newcommand{\bit}{\begin{itemize}}
\newcommand{\eit}{\end{itemize}}
\newcommand{\ben}{\begin{enumerate}}
\newcommand{\een}{\end{enumerate}}
\newcommand{\beq}{\begin{equation}}
\newcommand{\eeq}{\end{equation}}
\newcommand{\bea}{\begin{eqnarray*}} 
\newcommand{\eea}{\end{eqnarray*}}
\newcommand{\bpf}{\begin{proof}}
\newcommand{\epf}{\end{proof}\ms}
\newcommand{\bmt}{\begin{bmatrix}}
\newcommand{\emt}{\end{bmatrix}}
\newcommand{\ms}{\medskip}

\newcommand{\noi}{\noindent}
\usepackage{mathrsfs}
\newcommand{\ds}{\displaystyle}
\newcommand{\G}{\mathcal{G}}
\newcommand{\h}{\mathcal{H}}
\newcommand{\T}{\mathcal{T}}

\newcommand{\K}{\mathcal{K}}
\newcommand{\lolli}{\mathscr{L}}
\newcommand{\lp}{\left (}
\newcommand{\rp}{\right )}
\newcommand{\gpi}{\gamma_{P_{I}}}
\DeclareMathOperator{\spi}{sp}
\DeclareMathOperator{\pn}{pn}

\usepackage{tkz-euclide}
\usepackage[textsize=scriptsize,backgroundcolor=blue!20,linecolor=blue,bordercolor=blue]{todonotes}
\definecolor{lightgrey}{gray}{0.5}

\title{Infectious power domination of hypergraphs}
\author{Beth Bjorkman \thanks{Iowa State University Department of Mathematics}}
\date{February 27, 2019}

\begin{document}
\maketitle

\begin{abstract} 
The power domination problem seeks to find the placement of the minimum number of sensors needed to monitor an electric power network. We generalize the power domination problem to hypergraphs using the infection rule from Bergen et al: given an initial set of observed vertices, $S_0$, a set $A\subseteq S_0$ may infect an edge $e$ if $A\subseteq e$ and for any unobserved vertex $v$, if $A\cup \{v\}$ is contained in an edge, then $v\in e$. We combine a domination step with this infection rule to create \emph{infectious power domination}. We compare this new parameter to the previous generalization by Chang and Roussel. We provide general bounds and determine the impact of some hypergraph operations.
\end{abstract}


\section{Overview}\label{sec:overview}

The power domination problem seeks to find the placement of the mimimum number of sensors (called Phase Measurement Units or PMUs) needed to monitor an electric power network. In \cite{histpowdom}, Haynes et al. defined the power domination problem in graph theoretic terms by placing PMUs at a set of initial vertices and then applying observation rules to the vertices and edges of the graph. These observation rules consist of an initial domination step followed by what is now called the zero forcing process \cite{brueniheath05}, \cite{bfffhvw18}. Zero forcing is a graph propagation process that has its roots in determining the maximum nullity of the family of real symmetric matrices associated with the graph \cite{AIMminrank} and independently in control of quantum systems \cite{quantum}. 

Zero forcing has been generalized to hypergraphs in several different ways.  Bergen et. al defined the infection number of a hypergraph in \cite{infection18} to generalize the zero forcing process to hypergraphs. In \cite{hoghyp18}, Hogben defines the zero forcing number of a hypergraph based on the skew symmetric zero forcing number of a graph and the maximum nullity of a family of hypermatrices.  Chang and Roussel define $k$-power domination for hypergraphs in \cite{kpower15}, which is a power domination process when $k=1$. From this rule in \cite{kpower15}, Hogben also defines the power domination zero forcing number in \cite{hoghyp18}. 

Just as a power domination rule can be used to define a zero forcing process for hypergraphs, a zero forcing process can be used to define a power domination process for hypergraphs. Hogben's zero forcing number is less useful for power domination, as the zero forcing number of a hypergraph can be zero, which eliminates the real world connection to sensor placement. 

Our premise is to use Bergen et.~al's definition of infection \cite{infection18} to define infectious power domination of hypergraphs and compare this to the definition of power domination of hypergraphs introduced by Chang and Roussel. In Section \ref{sec:prelim} we will review preliminary definitions from past work. Then in Section \ref{sec:infpowdom} we formally define infectious power domination and compare it to both the infection number and the power domination number. In Section \ref{sec:genbounds} we determine general bounds for the power domination number and by extension the infectious power domination number. Section \ref{sec:determination} consists of results for the infectious power domination number including hypergraphs which have infectious power domination number one and hypertrees. Section \ref{sec:operations} determines bounds for the infectious power domination number for various hypergraph operations such as edge/vertex removal, linear sum, Cartesian products, and weak coronas. We make concluding remarks in Section \ref{sec:conclusion}.

\section{Preliminaries}\label{sec:prelim} 

We use Bretto's \emph{Hypergraph Theory} \cite{bretto} as a reference for hypergraph notation and definitions.

A \emph{hypergraph}, $\h=(V(\h),E(\h))$, is a set of vertices $V(\h)$ along with a set of edges $E(\h)$ so that $E(\h)$ is a subset of the power set of $V(\h)$. In the case that there is a constant $k$ such that $|e| =k$ for all $e\in E(\h)$,  we say that $\h$ is \emph{k-uniform} and denote such a hypergraph by $\h^{(k)}$. 

A \emph{path} in a hypergraph $\h$ is a sequence of vertices and edges $v_1,e_1,v_2,e_2,\ldots,v_\ell,e_\ell,v_{\ell+1}$ so that the $v_i$ are distinct vertices, the $e_i$ are distinct edges, and $v_i\in e_i$ for all $i$. We say that the path  $v_1,e_1,v_2,e_2,\ldots,v_\ell,e_\ell,v_{\ell+1}$ is a \emph{path from $v_1$ to $v_{\ell+1}$} and has \emph{length} $\ell$. A hypergraph is said to be \emph{connected} if there is a path from any vertex to any other vertex.  We say that a hypergraph $\h$ is \emph{reduced} if for all distinct edges $e,e'\in E$, $e\not\subseteq e'$ and $e'\not\subseteq e$; that is, no edge is contained in another edge. We may reduce a given hypergraph by removing every edge that is a proper subset of another edge. Throughout what follows, we will consider only hypergraphs with at least one edge that are reduced. 

The \emph{closed neighborhood} of a vertex $a\in V$ is $N[a]=  \ds \bigcup_{a\in e\in E} e$. The \emph{(open) neighborhood} of $a\in V$ is $N(a)= N[a]\setminus \{a\}$ and an element of $N(a)$ is called a \emph{neighbor} of $a$. The \emph{degree} of a vertex $a\in V$, denoted $\deg(a)$, is the number of edges that contain $a$. If $\deg(a)=0$, that is, $a$ is not contained in any edge, we say that $a$ is an \emph{isolated vertex}. An edge consisting of exactly one vertex is called a \emph{loop}. If vertices $a$ and $b$ are neighbors, we say that $a$ is \emph{adjacent} to $b$. When vertex $a$ is contained in edge $e$ we say that $e$ is \emph{incident} to $a$.

An \emph{induced subhypergraph} $\h'$ of a hypergraph $\h=(V,E)$ has a vertex set $V'\subseteq V$ and the edge set is $E'=\{e_i \cap V'  \neq \varnothing : e_i \in E, \mbox{ and either } e_i \text{ is a loop or } |e_i\cap V'|\geq 2\}$.  In this case, we say that $V'$ \emph{induces} the subhypergraph $\h'$. Note that if $\h$ is a $k$-uniform hypergraph that $\h'$ need not be uniform.


A \emph{dominating set} of a hypergraph $\h$ is a set of vertices $D\subseteq V(\h)$ so that for every vertex $v\in V\setminus D$, there exists an edge $e\in E(\h)$ for which $v\in e$ and $e\cap D\neq \varnothing$ \cite{henninglowenstein12}. That is, a dominating set is $D\subseteq V(\h)$ so that $V=\cup_{d\in D} N[d]$. The size of a minimum dominating set of $\h$ is called the \emph{domination number} of $\h$ and is denoted by $\gamma(\h)$.

\begin{defn}\rm{\cite{infection18}}\label{def:infection}
The \emph{infection rule} is defined so that a nonempty set $A$ of infected vertices can infect the vertices in an edge $e$ if:
\ben
\item $A\subseteq e$, and
\item if $v$ is an uninfected vertex such that $A\cup \{v\}$ is a subset of some edge in the hypergraph, then $v\in e$.
\een
\end{defn}

An initial set of infected vertices $S_0$ is called an \emph{infection set} if after repeated application of the infection rule all vertices become infected. The size of a minimum infection set is called the \emph{infection number} of the hypergraph $\h$ and is denoted by $I(\h)$.

\begin{defn}\rm{\cite{kpower15}}\label{def:powdomprocess}
The \emph{1-power domination process} consists of an initial subset of the vertices, $S_0$, called the \emph{power dominating set}  and the observation rules
\ben
\item[a.] A vertex in the power dominating set observes itself and all of its neighbors
\item[b.] If all unobserved neighbors of an observed vertex $v$ are in one edge incident to $v$, then these unobserved vertices become observed as well.
\een
We refer to step a as the \emph{domination step} and each repetition of b as an \emph{observation step}. 
\end{defn}

In this case, we say that a set of vertices $A$ \emph{observes} a set of vertices $B$ if $A$ causes $B$ to become observed. A \emph{power dominating set} of a hypergraph $\h$ is an initial set so that every vertex in $\h$ is observed at the termination of the 1-power domination process. The \emph{power domination number} of a hypergraph $\h$, denoted $\gamma_P(\h)$, is the minimum cardinality of a power dominating set of $\h$.  This is the same as a 1-power dominating set and the 1-power domination number as defined in \cite{kpower15} (originally denoted by $\gamma_P^1(\h)$).


\section{Infectious power domination}\label{sec:infpowdom}

We can now generalize power domination to hypergraphs based on the definition of infection in \cite{infection18}. 

\begin{defn}\label{def:infpowdom} Suppose $\h = (V,E)$ is a hypergraph. The \emph{infectious power domination process} on $\h$ with initial set $S_0\subseteq V$ is 
\ben
\item $S=\ds \bigcup_{v\in S_0} N[v]$.


\item While there exists a nonempty $A\subseteq S$ so that $A$ can infect the vertices in an edge $e$ using Definition \ref{def:infection}, add the vertices of $e$ to $S$.
\een
\end{defn}
An \emph{infectious power dominating set} of a hypergraph $\h$ is an initial set $S_0$ such that every vertex in $\h$ is in $S$ after the termination of the infectious power domination process. The \emph{infectious power domination number} of a hypergraph $\h$ is the minimum cardinality of an infectious power dominating set of $\h$, which we denote by $\gpi(\h)$. We say that a vertex in $S$ is \emph{infected} and that a set of vertices $A$ \emph{infects} a set of vertices $B$ if $A$ causes $B$ to join $S$. We call step 1 the \emph{domination step} and each repetition of step 2 an \emph{infection step}.


As an infection set can infect a graph without needing the power domination step, such a set is also an infectious power dominating set. Thus we have the following observation.

\begin{obs}\label{obs:gpileqinf} For any hypergraph $\h$, $\gpi(\h)\leq I(\h)$. \end{obs}

Chang and Roussel's definition in \cite{kpower15}, restated in Definition \ref{def:powdomprocess}, is equivalent to  Definition \ref{def:infpowdom} with the restriction that $A$ must always be a single vertex.  The next proposition is immediate and is also an easy consequence of Theorem 2.4 in \cite{hoghyp18}.

\begin{prop}\label{prop:leqpower}
For any hypergraph $\h$, $\gpi(\h) \leq \gamma_P(\h)$. 
\end{prop}

\noi The inequality in Proposition \ref{prop:leqpower} is not an equality, as shown in Example \ref{ex:linear}.

\begin{figure}[!ht]\caption{A hypergraph $\h^{(3)}$ with $\gpi(\h^{(3)})<\gamma_{P}(\h^{(3)})$.}\label{fig:lindif} 
\begin{center}
\begin{tikzpicture}
\foreach \a in {1,2,3}
    {\node[circle,fill=black,scale=0.5] (u\a) at ({\a*120-30}:1.25){};}
\foreach \a in {4,5,6}
    {\node[circle,fill=black,scale=0.5] (u\a) at ({\a*120+30}:.8){};}    
\foreach \a in {7,8,9}
    {\node[circle,fill=black,scale=0.5] (u\a) at ({\a*120-30}:2.25){};}
\foreach \a in {9,10,11}
    {\node[circle,fill=black,scale=0.5] (u\a) at ({\a*120-30}:3.25){};}
\node[left] (u1) at ({1*120-30}:1.25){3};
\node[left] (u2) at ({2*120-30}:1.25){9};
\node[left] (u3) at ({3*120-30}:1.25){5};
\node[left] (u4) at ({4*120+30}:.8){12};
\node[left] (u5) at ({5*120+30}:.8){8};
\node[right] (u6) at ({6*120+30}:.8){4};
\node[left] (u7) at ({7*120-30}:2.25){2};
\node[left] (u8) at ({8*120-30}:2.25){10};
\node[right] (u9) at ({9*120-30}:2.25){6};
\node[left] (u10) at ({10*120-30}:3.25){1};
\node[left] (u11) at ({11*120-30}:3.25){11};
\node[right] (u12) at ({12*120-30}:3.25){7};
\draw[rotate=30] (-2.25,0) ellipse (1.75cm and 0.75 cm);    
\draw[rotate=90] (2.25,0) ellipse (1.75cm and 0.75 cm);    
\draw[rotate=150] (-2.25,0) ellipse (1.75cm and 0.75 cm);    
\draw[rotate=120] (0,-0.66) ellipse (1.75cm and 0.75 cm);    
\draw[rotate=180] (0,0.66) ellipse (1.75cm and 0.75 cm);    
\draw[rotate=60] (0,0.66) ellipse (1.75cm and 0.75 cm);    
\end{tikzpicture}
\end{center}
\end{figure}
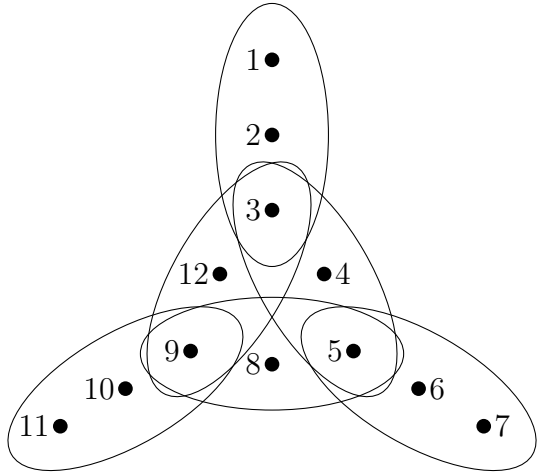

\begin{ex}\label{ex:linear}
For the hypergraph $\h^{(3)}$ in Figure \ref{fig:lindif}, $\gpi(\h^{(3)})=1 < \gamma_{P}(\h^{(3)})=2$. 

By symmetry, we need only check $\{2\}$, $\{3\}$, and $\{4\}$ as possible infectious power dominating sets or power dominating sets of size 1. 
\bit

\item $S_0=\{2\}$: 2 observes 1 and 3. Then 1 and 2 have no unobserved neighbors and 3 has unobserved neighbors in $\{3,12,9\}$ and $\{3,4,5\}$, and no other subset of $\{1,2,3\}$ is contained edge with unobserved vertices so no observation (or infection) step can occur. 
\item $S_0=\{4\}$: 4 observes $\{3,4,5\}$. Vertex 3 is in $\{1,2,3\}$ and $\{3,12,9\}$. Vertex 5 is in $\{5,6,7\}$ and $\{9,4,5\}$. Thus no more subsets of the observed vertices can observe an edge, nor can an infection step occur.
\item $S_0=\{3\}$: 3 observes 1, 2, 4, 5, 9, and 12. The only observed vertices adjacent to unobserved vertices are 5 and 9. Vertex 5 is contained in $\{9,8,5\}$ and $\{5,6,7\}$ so cannot observe an edge by itself. Similarly, $\{9\}$ cannot observe an edge. Thus no observation step can occur and so $\{3\}$ is not a power dominating set. However, $\{9,5\}$ can infect $\{9,8,5\}$. Then 5 infects $\{5,6,7\}$ and 9 infects $\{9,10,11\}$. Thus $\gpi(\h^{(3)}) = 1$.
\eit

For the power domination number, no one vertex is a power dominating set and so $\gamma_P(\h^{(3)}) > 1$. There is a power dominating set of size two: let $S_0=\{3,5\}$. In the domination step, vertices 1, 2, 4, 6, 7, 8, 9, and 12 become observed. Then $\{9\}$ observes $\{9,10,11\}$. Therefore $\{3,5\}$ is a power dominating set of $\h^{(3)}$ and $\gamma_{P}(\h^{(3)})=2$.
\end{ex}

Next we present examples showing that the infection number can be drastically different from the power domination number and infectious power domination number, particularly in the case of a $k$-uniform hypergraph when $n$ is large and $k$ is small. 
The complete $k$-uniform hypergraph, $\K_{n}^{(k)}$, is the $k$-uniform hypergraph with vertex set $\{1,2,\ldots,n\}$ and edge set all $k$-sets of the vertex set. 

\begin{prop} {\rm\cite[Lemma 3.1]{infection18}}
$I\lp\K_n^{(k)} \rp = n-k+1$.
\end{prop}

The next result is immediate as $N[v]= V(\K_n^{(k)})$ for any $v\in V(\K_n^{(k)})$.

\begin{prop}\label{prop:kcomplete}
$\gpi \lp\K_n^{(k)} \rp= \gamma_{P} \lp\K_n^{(k)} \rp=\gamma \lp\K_n^{(k)} \rp= 1$.
\end{prop}

For a less trivial example of the potential gap between the infection number and the infectious power domination number, we consider the \emph{complete $k$-partite hypergraph}, $\K_{n_1,n_2,\ldots,n_k}^{(k)}$, which is the hypergraph that has its vertex set partitioned into $k$ disjoint parts $V_1,\ldots,V_k$ where $|V_i|=n_i$. The edge set is the set of all $k$-sets with exactly one element from each $V_i$. Note that $\K_{n_1,n_2,\ldots,n_k}^{(k)}$ is $k$-uniform by definition.

\begin{prop} {\rm\cite[Lemma 3.4]{infection18}}
$I\lp\K_{n_1,\ldots,n_k}^{(k)}\rp = n_1+n_2+\cdots+n_k -k$. 
\end{prop}

\begin{prop}\label{prop:unikpartite} For the complete $k$-partite hypergraph, we have the following:
\bea
\ds \gamma \lp\K_{n_1,\ldots,n_k}^{(k)}\rp &=& \left\{\begin{array}{cl} 
        1 & \ds \min_{1\leq \ell\leq k} (n_\ell) = 1 \\ 
        2 & \textup{otherwise} \end{array} \right. \\
&   \\
&   \\
\ds \gamma_P \lp\K_{n_1,\ldots,n_k}^{(k)}\rp &=& \left\{\begin{array}{cl} 
        1 & \ds \min_{1\leq \ell\leq k} (n_\ell) = 1 \\
        1 & \ds \min_{1\leq \ell\leq k} (n_\ell) = 2, k=2\\
        2 & \ds \min_{1\leq \ell\leq k} (n_\ell) = 2, k>2 \textup{ or } \min_{1\leq \ell\leq k}(n_\ell)\geq 3  \end{array} \right.\\
&   \\
&   \\
\ds \gpi \lp\K_{n_1,\ldots,n_k}^{(k)}\rp &=& \left\{\begin{array}{cl} 
        1 & \ds \min_{1\leq \ell\leq k} (n_\ell) \leq 2 \\ 
        2 & \textup{otherwise} \end{array} \right.. \\ 
\eea
\end{prop}

\begin{proof}
If any $n_i=1$ then the sole vertex in $V_i$ is adjacent to every other vertex and so $\gpi\lp\K_{n_1,\ldots,n_k}^{(k)}\rp\leq \gamma_P\lp\K_{n_1,\ldots,n_k}^{(k)}\rp\leq \gamma\lp\K_{n_1,\ldots,n_k}^{(k)}\rp=1$. 

If all $n_\ell\geq 2$, then for any vertex $v_i\in V_i$, we see that $v_i$ has at least one non-neighbor and so $\gamma\lp\K_{n_1,\ldots,n_k}^{(k)}\rp\geq 2$. Consider $S_0=\{v,u\}$ where $v$ and $u$ are adjacent (in different parts). Then $v$ is adjacent to all non-neighbors of $u$ and $u$ is adjacent to all non-neighbors of $v$; therefore, $\gpi\lp\K_{n_1,\ldots,n_k}^{(k)}\rp\leq \gamma_P\lp\K_{n_1,\ldots,n_k}^{(k)}\rp\leq\gamma \lp \K_{n_1,\ldots,n_k}^{(k)} \rp =2.$

Next suppose that some $n_i = 2$ and $k=2$. Without loss of generality, $n_1=2$ with that $V_1=\{v_1,w_1\}$. Consider $S_0=\{v_1\}$. After the domination step, $w_1$ is the only unobserved vertex. Let $v_2\in V_2$. Then $\{v_2\}\cup \{w_1\} \in E\lp\K_{n_1,\ldots,n_k}^{(k)}\rp$ and so $\{v_2\}$ observes $\{w_1\}$. In this case, $\gamma_P\lp\K_{n_1,\ldots,n_k}^{(k)}\rp = 1$.

Let $ \min_{1\leq \ell\leq k} (n_\ell) = 2$ and $k>2$. Let $v_i\in V_i$ and $S_0=\{v_i\}$. After the domination step, only the non-neighbors of $v_i$ are unobserved; let $w_i$ be one of these non-neighbors. Let $v_j\in V_j$ with $j\neq i$ and  let $v_k,v_k' \in V_k$ with $k\neq i,j$. Then $\{v_j\}\cup \{v_i\}$ is contained in an edge $e$ with $\{v_i,v_j,v_k\}\in e$ and in an edge $e'$ with $\{v_i,v_j,v_k'\}\in e'$. Thus there is not a power dominating set of size 1 and so $2\leq \gamma_P\lp\K_{n_1,\ldots,n_k}^{(k)}\rp\leq\gamma \lp \K_{n_1,\ldots,n_k}^{(k)} \rp =2.$

If some $n_i=2$ let one of these two vertices be $S_0$. After the domination step, only the remaining vertex in $V_i$, say $v_i$, is uninfected. Let $e$ be any edge containing $v_i$ and consider $A=e\setminus\{v_i\}$. As $v_i$ is the only uninfected vertex, and $A\cup \{v_i\} = e$, no other edge contains $A\cup\{v_i\}$. Thus, $A$ infects $e$. Therefore, $\gpi\lp\K_{n_1,\ldots,n_k}^{(k)}\rp=1$.

Finally, if $\min_{1\leq \ell\leq k} (n_\ell) \geq 3$, then choosing one vertex $v_i\in V_i$ infects all of the neighbors of $v_i$ in the domination step and the uninfected vertices consist of the $n_i-1\geq 2$ non-neighbors of $v$. Let $w_i,w_i'$ be two of these non-neighbors.
Suppose $A$ is a set of infected vertices such that $A\cup \{w_i\} \subseteq e$. Then $e'=(e\setminus \{w_i\} )\cup \{w_i'\}$ is also an edge. However, $A\cup \{w_i'\} \subseteq e'$ and $w_i'\not\in e$. Thus the non-neighbors of $v_i$ cannot be infected. Hence $\gpi\lp \K_{n_1,\ldots,n_k}^{(k)} \rp \geq 2$. Therefore, $2\leq \gpi \lp \K_{n_1,\ldots,n_k}^{(k)} \rp \leq \gamma_P\lp\K_{n_1,\ldots,n_k}^{(k)}\rp\leq \gamma \lp \K_{n_1,\ldots,n_k}^{(k)} \rp =2. $
\end{proof}

Again we see that if the number of vertices is large and $k$ is small, the discrepancy between the infection number and the infectious power domination number may be large.


Both generalizations of power domination to hypergraphs reduce to the power domination problem for graphs when $\h$ is 2-uniform, i.e. $\h$ is a graph (Prop. 1.1 in \cite{infection18} and page 1097 in \cite{kpower15}). However, Chang and Roussel's definition focuses on allowing one vertex to observe others whereas infectious power domination utilizes the fact that there may be multiple observed vertices in an edge which can be used to observe the edge. Using multiple vertices to observe an edge may be more natural as a model for physical problems, as this represents using measurements from multiple sensors.

\section{General bounds}\label{sec:genbounds}

In this section, we give bounds for the power domination number and infectious power domination number in terms of the degrees of the vertices, the number of edges, the size of the edges, and the number of vertices. 

\begin{prop}\label{prop:degbound}
Let $\h$ be a connected hypergraph. If $\h$ has at least one vertex of degree at least $3$, then $\gpi(\h)\leq \gamma_{P}(\h) \leq \left| \left\{ v\in V(\h) : \deg(v)\geq 3 \right\}\right|$. If $\deg(v) \leq 2$ for all vertices $v$ of $\h$, then $\gpi(\h)=\gamma_{P}(\h)=1$.
\end{prop}

\begin{proof}
First, assume that $\h$ has at least one vertex of degree at least $3$. Let $S_0$ be the set of vertices with degree at least 3. After the domination step, any remaining unobserved vertex has degree at most two. Moreover, each vertex in $S_1 = N(S_0)\setminus S_0$ has degree at most two. One of these two edges contains only observed vertices and so each vertex in $S_1$ can observe the precisely one edge containing observed vertices incident to it. Each of these newly observed vertices are now in at most one edge containing unobserved vertices and so can observe that remaining edge. This continues until the entire graph is observed.

On the other hand, if $\deg(v)\leq 2$ for all $v\in V(\h)$, then select one vertex. After the domination step, the entire graph will become observed in the same way as in the previous case.
\end{proof}

\begin{prop}\label{prop:edgebound}
For any connected hypergraph $\h$ with at least two edges, we have $$\gpi(\h)\leq \gamma_{P}(\h) \leq |E(\h)|-1.$$ This bound is tight.
\end{prop}

\begin{proof}
If we select one vertex from each edge of $\h$ save one, then in the domination step we observe all but at most one edge. This edge is then the unique edge containing unobserved vertices and so can be observed via the observation step as the $\h$ is connected.

It follows that any connected hypergraph $\h$ with exactly two edges has $\gpi(\h)=1$ and so we see that the bound is tight. 
\end{proof}

The next upper bound is similar to Proposition 1.2 in \cite{infection18}.

\begin{prop}\label{prop:edgesizebound}
Let $\h$ be a nontrivial hypergraph on $n$ vertices with at least two edges and let $k$ be the size of the largest edge in $\h$. Then $\gpi(\h)\leq \gamma_{P}(\h) \leq n-k$. This bound is tight.
\end{prop}

\begin{proof}
Let $e$ be the largest edge of $\h$. Let $S_0=V(\h)\setminus e$. As $\h$ is connected with at least two edges, at least one vertex, say $v$, in $e$ is adjacent to a vertex which is not in $e$, i.e. $v$ is adjacent to a vertex in $S_0$. Thus in the domination step $v$ becomes observed. Then all of the unobserved neighbors of $v$ are contained in $e$ and so they become observed in the observation step. 

Consider the hypergraph $\h$ consisting of $n$ vertices with two edges: one edge containing $n-1$ vertices and the other containing one vertex from the first edge and the remaining vertex. The vertex in the intersection of the two edges infects all of the vertices in the domination step. Thus $\gpi(\h)=1=n-(n-1)$.
\end{proof}

As power domination and infectious power domination consist of a domination step with the addition of the observation or infection step, we have the following  observation.

\begin{obs} \label{obs:gpileqdom}
For any hypergraph $\h$, $\gpi(\h)\leq \gamma_{P}(\h) \leq \gamma(\h)$.
\end{obs}

Thus, we may utilize the following domination bound from \cite{henninglowenstein12}.

\begin{prop}\rm{\cite[Theorem 2]{henninglowenstein12}}  \label{prop:nover3bounddom}
If $\h$ is a hypergraph with all edges of size at least three and no isolated vertex then $\gamma(\h)\leq \frac{|V(\h)|}{3}$.
\end{prop}

We immediately obtain the following corollary from Observation \ref{obs:gpileqdom} and Proposition \ref{prop:nover3bounddom}.

\begin{cor}\label{cor:nover3bound}
If $\h$ is a hypergraph with all edges of size at least three and no isolated vertex then $\gpi(\h)\leq \gamma_P(\h) \leq \frac{|V(\h)|}{3}$.
\end{cor}

Corollary \ref{cor:nover3bound} does not utilize the observation (or infection) step. For graphs, the well known domination upper bound of $\frac{|V(G)|}{2}$ was improved for power domination  to $\frac{|V(G)|}{3}$ in \cite{zhaokangchang06}. We conjecture the following similar result.

\begin{conj}\label{conj:noverfour}
For any connected hypergraph $\h$ on at least 4 vertices with $|e|\geq 3$ for all $e\in E(\h)$, then $\gpi(\h)\leq \gamma_{P}(\h)\leq \frac{|V(\h)|}{4}$.
\end{conj}

The \emph{$X$-private neighborhood} of a vertex $v\in X$ is $\pn(v,X)=N(v)\setminus \bigcup_{x\in X\setminus\{v\}} N[x]$, a variant of the definition in \cite{zhaokangchang06}. The members of $\pn(v,X)$ are the \emph{$X$-private neighbors of $v$}. Zhao, Kang, and Chang's proof from \cite{zhaokangchang06} is a counting argument in which they find two $S_0$-private neighbors 
for each vertex in the power dominating set, giving the bound via the inequality $|V(G)|\geq |S_0|+2|S_0|$. 

However, this strategy does not translate to hypergraphs. Consider the hypergraph $\lolli_1^{(3)}$ shown in Figure \ref{fig:lollipop}. We will use $\lolli_1^{(3)}$ to build a family of hypergraphs which achieves the bound in Conjecture \ref{conj:noverfour}, but in which any minimum power dominating set contains a vertex with at most two private neighbors.

\begin{figure}[!hbt]\caption{ Shown is $\lolli_1^{(3)}$. The edges are $\{w,x,v\}$, $\{w,x,y\}$, and $\{w,x,z\}$. }\label{fig:lollipop} 
\begin{center}
\begin{tikzpicture}[scale=0.8]
    \node (v1) at (0,0) {};
    \node (v2) at (0,4) {};
    \node (v3) at (0,8) {};
    \node (v4) at (-2,6) {};
    \node (v5) at (2,6) {};
    \draw (0,8.5) arc (90:270:2.5);
    \draw (0,7.5) arc (90:270:1.5);
    \draw (0,7.5) arc (90:270:-0.5);
    \draw (0,3.5) arc (90:270:-0.5);
    \draw (0,3.5) arc (90:270:-2.5);
    \draw (0,4.5) arc (90:270:-1.5);
    \draw (0,8.5) arc (90:270:0.5);
    \draw (0,4.5) arc (90:270:0.5);
    \draw (-0.5,0) arc (0:180:-0.5) -- (0.5,8) -- (0.5,8) arc (0:180:0.5) -- (-0.5,0) ;
    \foreach \v in {1,2,...,5} {
        \fill (v\v) circle (0.1);
    }
    \fill (v1) circle (0.1) node [above] {$v$};
    \fill (v2) circle (0.1) node [below] {$x$};
    \fill (v3) circle (0.1) node [below] {$w$};
    \fill (v4) circle (0.1) node [below] {$y$};
    \fill (v5) circle (0.1) node [below] {$z$};
\end{tikzpicture}
\end{center}
\end{figure}
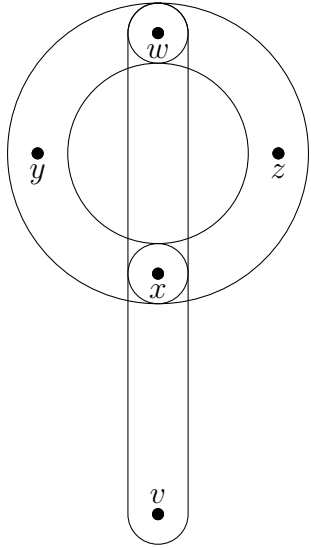

The hypergraph $\lolli_q^{(3)}$ is constructed by taking $q$ copies of $\lolli_1^{(3)}$, say $\lolli_{1,i}^{(3)}$, with vertex sets $\{v_i,x_i,y_i,z_i,w_i\}$ for each $1\leq i \leq q$. Identify $w_i$ with $v_{i+1}$ for $1\leq i \leq q-1$ and $w_q$ with $v_1$. Then 
\[ \lolli_q^{(3)} =\lp \bigcup_{1\leq i \leq q} \{x_i,y_i,z_i,w_i\}, \bigcup_{1\leq i \leq q} E\lp\lolli_{1,i}^{(3)}\rp\rp.\]
 For an example, $\lolli^{(3)}_3$ is shown in Figure \ref{fig:lollipop3}. We construct the family $\lolli = \ds\{ \lolli_q^{(3)} : q\geq 2 \ds\}$.

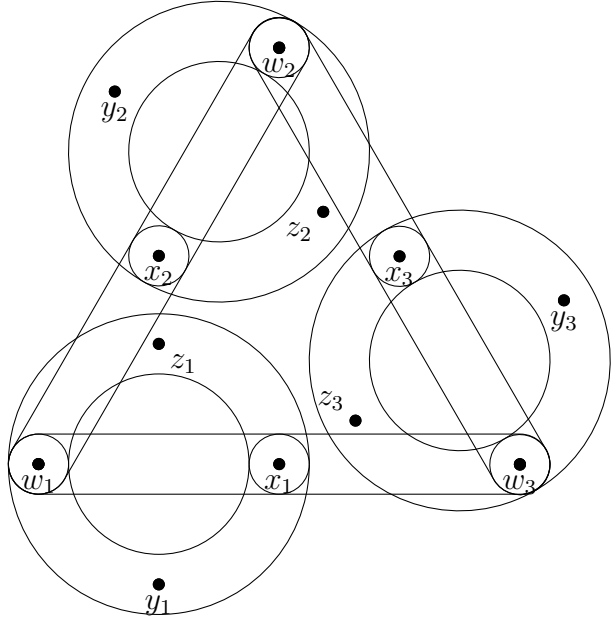
\begin{figure}[!ht]\caption{Shown is $\lolli^{(3)}_3$, with the identifications $w_1=v_2$, $w_2=v_3$, and $w_3=v_1$.}\label{fig:lollipop3} 
\begin{center}
\begin{tikzpicture}[scale=0.8]

\begin{scope}[rotate=90]
    \node (v1) at (0,0) {};
    \node (v2) at (0,4) {};
    \node (v3) at (0,8) {};
    \node (v4) at (-2,6) {};
    \node (v5) at (2,6) {};
    \draw (0,8.5) arc (90:270:2.5);
    \draw (0,7.5) arc (90:270:1.5);
    \draw (0,7.5) arc (90:270:-0.5);
    \draw (0,3.5) arc (90:270:-0.5);
    \draw (0,3.5) arc (90:270:-2.5);
    \draw (0,4.5) arc (90:270:-1.5);
    \draw (0,8.5) arc (90:270:0.5);
    \draw (0,4.5) arc (90:270:0.5);
    \draw (-0.5,0) arc (0:180:-0.5) -- (0.5,8) -- (0.5,8) arc (0:180:0.5) -- (-0.5,0) ;
    \foreach \v in {1,2,...,5} {
        \fill (v\v) circle (0.1);
    }
    \fill (v1) circle (0.1) node [above] {};
    \fill (v2) circle (0.1) node [below] {$x_1$};
    \fill (v3) circle (0.1) node [below] {$w_1$};
    \fill (v4) circle (0.1) node [below] {$y_1$};
    \fill (v5) circle (0.1) node [below right] {$z_1$};
\end{scope}

\begin{scope}[xshift=-4cm,yshift=6.92cm,rotate=-150]
    \node (v1) at (0,0) {};
    \node (v2) at (0,4) {};
    \node (v3) at (0,8) {};
    \node (v4) at (-2,6) {};
    \node (v5) at (2,6) {};
    \draw (0,8.5) arc (90:270:2.5);
    \draw (0,7.5) arc (90:270:1.5);
    \draw (0,7.5) arc (90:270:-0.5);
    \draw (0,3.5) arc (90:270:-0.5);
    \draw (0,3.5) arc (90:270:-2.5);
    \draw (0,4.5) arc (90:270:-1.5);
    \draw (0,8.5) arc (90:270:0.5);
    \draw (0,4.5) arc (90:270:0.5);
    \draw (-0.5,0) arc (0:180:-0.5) -- (0.5,8) -- (0.5,8) arc (0:180:0.5) -- (-0.5,0) ;
    \foreach \v in {1,2,...,5} {
        \fill (v\v) circle (0.1);
    }
    \fill (v1) circle (0.1) node [above] {};
    \fill (v2) circle (0.1) node [below] {$x_3$};
    \fill (v3) circle (0.1) node [below] {$w_3$};
    \fill (v4) circle (0.1) node [below] {$y_3$};
    \fill (v5) circle (0.1) node [above left] {$z_3$};
\end{scope}

\begin{scope}[xshift=-8cm,rotate=-30]
    \node (v1) at (0,0) {};
    \node (v2) at (0,4) {};
    \node (v3) at (0,8) {};
    \node (v4) at (-2,6) {};
    \node (v5) at (2,6) {};
    \draw (0,8.5) arc (90:270:2.5);
    \draw (0,7.5) arc (90:270:1.5);
    \draw (0,7.5) arc (90:270:-0.5);
    \draw (0,3.5) arc (90:270:-0.5);
    \draw (0,3.5) arc (90:270:-2.5);
    \draw (0,4.5) arc (90:270:-1.5);
    \draw (0,8.5) arc (90:270:0.5);
    \draw (0,4.5) arc (90:270:0.5);
    \draw (-0.5,0) arc (0:180:-0.5) -- (0.5,8) -- (0.5,8) arc (0:180:0.5) -- (-0.5,0) ;
    \foreach \v in {1,2,...,5} {
        \fill (v\v) circle (0.1);
    }
    \fill (v1) circle (0.1) node [above] {};
    \fill (v2) circle (0.1) node [below] {$x_2$};
    \fill (v3) circle (0.1) node [below] {$w_2$};
    \fill (v4) circle (0.1) node [below] {$y_2$};
    \fill (v5) circle (0.1) node [below left] {$z_2$};
\end{scope}

\end{tikzpicture}
\end{center}
\end{figure}

\begin{prop}\label{prop:n4lollipops}
The family of hypergraphs $\lolli$ satisfies \[\ds \gpi(\lolli_q^{(3)}) = \gamma_P(\lolli_q^{(3)}) = \gamma(\lolli_q^{(3)}) = q = \frac{|V(\lolli_q^{(3)})|}{4}\] for all $\lolli_q^{(3)}\in\lolli$. Any minimum power dominating set for a member of $\lolli$, $S_0$, contains a vertex with at most two $S_0$-private neighbors.
\end{prop}

\begin{proof}
First we show $\gpi(\lolli_q^{(3)}) \geq q$. Suppose for contradiction that $\gpi(\lolli_q^{(3)}) < q$. Then there exists $j$ so that $S_0\cap \{x_j,y_j,z_j,w_j\} = \varnothing$. However, there is no way for $y_j$ to become infected as its only neighbors, $x_j$ and $w_j$, are also in an edge with the uninfected vertex $z_j$. This is a contradiction and so $\gamma_P(\lolli_q^{(3)}) \geq q$. This also implies that for all $1\leq i\leq q$, at least one of $\{x_i,y_i,z_i,w_i\}$ must be in any infectious power dominating set or any power dominating set. 

For equality, observe that $\{x_i : 1\leq i \leq q\}$ is a dominating set of $\lolli_q^{(3)}$ of size $q$. Therefore, $q \leq \gpi(\lolli_q^{(3)}) \leq \gamma_P(\lolli_q^{(3)}) \leq \gamma(\lolli_q^{(3)}) \leq q.$ 

Now consider a minimum infectious power dominating set of $\lolli_q^{(3)}$, $S_0$. We will show that $S_0$ contains a vertex with at most two private neighbors.

If for some $j$, $y_j\in S_0$, then as $y_j$ has only two neighbors, $S_0$ contains a vertex with at most two $S_0$-private neighbors. The same argument applies to $z_j$. Thus, we need only consider minimum power dominating sets consisting only of $w$ and $x$ type vertices and for all $1\leq i\leq q$, either $w_i$ or $x_i$ must be in $S_0$.

Let $x_i\in S_0$. Without loss of generality, $N(x_i)=\{y_i,z_i,w_i,w_{i-1}\}$. Either $w_{i+1}$ or $x_{i+1}$ must be in $S_0$ and $w_{i}$ is adjacent to both of these. Thus, $w_{i}\not\in \pn(x_i,S_0)$. Similarly, either $x_{i-1}$ or $w_{i-1}$ must be in $S_0$, so $w_{i-1}\not\in \pn(x_i,S_0)$. Hence $\pn(x_i, S_0)=\{y_i,z_i\}$. Finally, if $S_0=\{w_i : 1\leq i \leq p\}$, then $\pn(w_i,S_0) = \{y_i,z_i\}$ for all $i$, as $x_i$ is a common neighbor with $w_{i-1}$ and $x_{i+1}$ is a common neighbor with $w_{i+1}$. 

Therefore, $\lolli_q^{(3)}$ has at least one vertex in every minimum power dominating set which has at most two private neighbors.
\end{proof}

\section{Infectious power domination number for particular hypergraphs} \label{sec:determination} 



We have the following easy consequence of Observation \ref{obs:gpileqinf}.

\begin{obs}\label{obs:gpieq1}  For any hypergraph $\h$, if $I(\h)=1$, then $\gpi(\h)=1$. \end{obs}
Thus we obtain the several results directly from \cite{infection18}, after some definitions. A hypergraph is an \emph{interval hypergraph} if there is a linear ordering of the vertices so that each edge consists of consecutive vertices. A \emph{hypercycle} is a connected hypergraph with edge set $e_1,\ldots,e_\ell$ with $\ell \geq 4$ so that $e_i\cap e_j \neq \varnothing$ if and only if $i-j\equiv \pm 1 \mod \ell$, \cite{infection18}. 

\begin{obs}\label{obs:gpieq1examples}
\end{obs}
\ben
\item {\rm\cite[Prop. 2.1]{infection18}} \emph{If $E(\h)=\{V(\h)\}$, then $\gpi(\h)=1$.}
\item {\rm\cite[Prop 3.2]{infection18} } \emph{Let $\h$ be a $k$-hypergraph with $k\geq 3$. Then there exists a $k$-hypergraph $\h'$ such that $V(\h)\subseteq V(\h')$ and $E(\h) \subseteq E(\h')$ with $\gpi(\h)=1$.}
\item {\rm\cite[Lem. 4.3]{infection18}} \emph{If $\h$ is a connected interval hypergraph then $\gpi(\h)=1$. }
\item {\rm \cite[Prop. 4.5]{infection18}} \label{obs:gpieq1exhypercycle}\emph{ If $\h$ is a hypercycle with a vertex of degree 1, then $\gpi(\h)=1$.}
\een

In fact, Observation \ref{obs:gpieq1examples}.\ref{obs:gpieq1exhypercycle} is true without the restriction on vertex degrees, which follows from the next proposition.

A \emph{circular arc interval hypergraph} is a hypergraph $\h$ with $n$ vertices with a circular order of the vertices so that every edge is composed of consecutive vertices. Using this circular ordering of the vertices, the first end point  of each edge is unique. If $\h$ has $m$ edges, let the first end points be denoted by $v_1,v_2,\ldots, v_m$ with corresponding edges $e_1,\ldots,e_m$, in a similar way to \cite{hoghyp18}. When $\h$ is connected, we may choose an ordering so that $v_1, e_1, v_2, e_2, \cdots e_{m-1}, v_m$ forms a path. 


\begin{prop}
For any connected circular arc interval hypergraph $\h$, \[\gpi(\h)=\gamma_{P}(\h)=1.\]
\end{prop}

\begin{proof} 
 Let $S_0=\{v_1\}$, the left endpoint of $e_1$. After the domination step, each vertex of $e_1$ is observed. As $\h$ is connected, this means that $v_2$ is observed. Since $\h$ is a circular arc hypergraph, any edge containing $v_2$ but not containing $v_1$ must have a left endpoint that is after $v_1$. The first left endpoint after $v_1$ is $v_2$. As $v_2$ is the left endpoint of $e_2$, the only edge containing unobserved vertices and $v_2$ is $e_2$, because any other edge containing $v_2$ would also have to contain $v_1$ and so would have become observed in the domination step. Thus $v_2$ observes $e_2$. In a similar way, $v_3$ is in $e_2$ and so $v_3$ observes $e_3$. We continue this process until all edges are observed.
\end{proof}

A \emph{Berge cycle} in a hypergraph $\h$ is a sequence $C_m=(v_1,e_1,v_2,e_2,\ldots,v_m,e_m,v_1)$ in which the $v_i$ are distinct vertices and the $e_i$ are distinct edges. A \emph{hypertree} is a connected hypergraph which contains no Berge cycle. A hypergraph is \emph{linear} if distinct edges intersect in at most one vertex. Note that any hypertree is linear. To see this, consider if two edges $e_1$ and $e_2$ share vertices $v_1$ and $v_2$. Then $v_1 e_1 v_2 e_2 v_1$ is a Berge cycle.  A \emph{major vertex} is a vertex of degree at least 3 \cite{kpower15}. A \emph{spider} is a nonempty hypertree with at most one major vertex. A \emph{spider cover} of a hypertree $\T$ is a partition of $V(\T)$, $V_1,\ldots, V_\ell$, such that each subset induces a spider. The \emph{spider number} of a hypertree $\T$ is the minimum size of a spider cover of $\T$, denoted by $\spi(\T)$.

\begin{thm}{\rm \cite[Theorem 7]{kpower15}} \label{thm:powertrees}
For any hypertree $\T$, $\gamma_{P}(\T) = \spi(\T)$.
\end{thm}

As a direct result of Theorem \ref{thm:powertrees} and Proposition \ref{prop:leqpower}, we have the following proposition.

\begin{cor}
For any hypertree $\T$, $\gpi(\T) \leq\spi(\T)$.
\end{cor}

For the infectious power domination number, this bound is not an equality as shown in Example \ref{ex:hypertree}.

\begin{figure}[!h]\caption{A hypertree $\T^{(3)}$ with $\gpi(\T^{(3)})<\gamma_{P}(\T^{(3)})$.}\label{fig:hypertree} 
\begin{center}
\begin{tikzpicture}
\foreach \a in {1,2,...,7}
    {\node[circle,fill=black,scale=0.5] (u\a) at (\a-1,0){};}
\foreach \a in {7,8}
    {\node[circle,fill=black,scale=0.5] (u\a) at (0,\a-6){};}
\foreach \a in {9,10}
    {\node[circle,fill=black,scale=0.5] (u\a) at (0,-\a+8){};} 
\foreach \a in {11,12}
    {\node[circle,fill=black,scale=0.5] (u\a) at (6,\a-10){};}
\foreach \a in {13,14}
    {\node[circle,fill=black,scale=0.5] (u\a) at (6,-\a+12){};}   
\foreach \a in {11,12}
    {\node[circle,fill=black,scale=0.5] (u\a) at (2,\a-10){};}  
\foreach \a in {11,12}
    {\node[circle,fill=black,scale=0.5] (u\a) at (4,\a-10){};}  
\foreach \a in {1,2,...,7}
    {\node[left] (u\a) at (\a-1,0){\a};}
\foreach \a in {8,9}
    {\node[left] (u\a) at (0,\a-7){\a};}
\foreach \a in {10,11}
    {\node[left] (u\a) at (0,-\a+9){\a};} 
\foreach \a in {16,17}
    {\node[left] (u\a) at (6,\a-15){\a};}
\foreach \a in {18,19}
    {\node[left] (u\a) at (6,-\a+17){\a};}   
\foreach \a in {12,13}
    {\node[left] (u\a) at (2,\a-11){\a};}  
\foreach \a in {14,15}
    {\node[left] (u\a) at (4,\a-13){\a};}  
\draw (1,0) ellipse (1.5cm and .75cm);
\draw (3,0) ellipse (1.5cm and .75cm);
\draw (5,0) ellipse (1.5cm and .75cm);
\draw (0,1.1) ellipse (.75cm and 1.6cm);
\draw (0,-1.1) ellipse (.75cm and 1.6cm);
\draw (2,1.1) ellipse (.75cm and 1.6cm);
\draw (4,1.1) ellipse (.75cm and 1.6cm);
\draw (6,1.1) ellipse (.75cm and 1.6cm);
\draw (6,-1.1) ellipse (.75cm and 1.6cm);
\end{tikzpicture}
\end{center}
\end{figure}
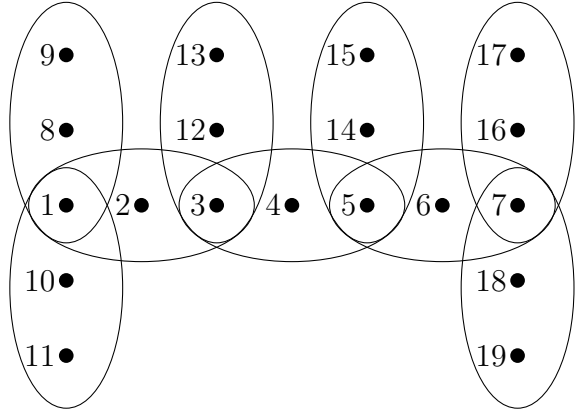

\begin{ex}\label{ex:hypertree}
For the hypertree $\T^{(3)}$ in Figure \ref{fig:hypertree}, $\gpi(\T^{(3)})=2 < 3 = \gamma_{P}(\T^{(3)})$.

We first show that there is no spider cover of $\T^{(3)}$ of size 2. The major vertices of $\T^{(3)}$ are $1,3,5,7$. Suppose for contradiction that we have a spider cover $V_1, V_2$, and without loss of generality, let $1\in V_1$.

First suppose that $1,3,5,7 \in V_1$. However, this means that at least two of \[\{8,9\},\{10,11\},\{12,13\},\{14,15\},\{16,17\},\{18,19\}\] are in $V_2$ and so we have a contradiction as $V_2$ must induce a connected hypergraph.

Next consider if $1,3,5\in V_1$ and $7\in V_2$.  However, then at least one of \[\{8,9\}, \{10,11\},\{12,13\},\{14,15\}\] is not in $V_1$. However, this means that $V_2$ is disconnected as $7\in V_2$. Ergo, we cannot have any three major vertices in either $V_1$ and similarly for $V_2$.

Thus we have at most two elements of $\{1,3,5,7\}$ in $V_1$. Note that $1,7\in V_1$ or $1,5\in V_1$  would imply that $3\in V_1$ as well because $V_1$ induces a connected hypergraph. Thus, we must have $1,3\in V_1$ and $5,7\in V_2$. If $4\in V_1$, then $12$ or one of $8,10$ must be in $V_2$, but this is a contradiction to the connectedness of the subhypergraph induced by $V_2$. Similarly, $4\not\in V_2$. Therefore, $\spi(\T^{3}) >2 $.

Observe that
\[\{\{1,2,8,9,10,11\}, \{3,4,5,12,13,14,15\}, \{6,7,16,17,18,19\}\}\]
is a spider cover of $\T^{(3)}$. Thus, $\spi (\T^{(3)})=3$. By Theorem \ref{thm:powertrees}, $\spi(\T^{(3)}) = \gamma_P(\T^{(3)}) = 3$.

We now consider possible infectious power dominating sets. By symmetry, for sets of size 1 we need only check $1, 2, 3, 4, 8$ and $12$.

\bit
\item $S_0=\{1\}$: $1$ infects $2,3,8,9,10,11$. Then vertex $3$ is the only infected vertex which has uninfected neighbors, however these neighbors occur in both edge $\{3,12,13\}$ and edge $\{3,4,5\}$, so no infection step can occur.
\item $S_0=\{2\}$: $2$ infects $1,3$. Then vertex $1$ is adjacent to uninfected vertices in both $\{1,8,9\}$ and $\{1,10,11\}$. Vertex $3$ is adjacent to uninfected vertices in both $\{3,4,5\}$ and $\{3,12,13\}$. No other subset of infected vertices is contained in an edge also containing uninfected vertices and so no infection step can occur.
\item $S_0=\{3\}$: $3$ infects $1,2,4,5,12,13$. However, $1$ is in both $\{1,8,9\}$ and $\{1,10,11\}$. Vertex $5$ is in both $\{5,14,15\}$ and $\{5,6,7\}$.
\item $S_0=\{4\}$ or $S_0=\{12\}$: Since $\{3\}$ is not an infectious power dominating set, neither is $\{4\}$ or $\{12\}$ as $N(4), N(12) \subset N(3)$. 
\item $S_0=\{8\}$: Since $\{1\}$ is not an infectious power dominating set, neither is $\{8\}$ as $N(8)\subset N(1)$. 
\eit
Thus there is no infectious power dominating set of size 1. Next consider $S_0=\{1,7\}$. After the domination step, $1, 2, 3, 5, 6, 7, 8, 9, 10, 11, 16, 17, 18,$ and $19$ are infected. Then $\{3,5\}$ infects $4$. Finally, $\{3\}$ infects $\{12,11\}$ and $\{5\}$ infects $\{14,15\}$. Thus $\gpi(\T^{(3)})=2$.

Therefore, we have $\gpi(\T^{(3)}) = 2 < \spi(\T^{(3)})=3$.

\end{ex}

\section{Hypergraph operations}\label{sec:operations}

\subsection{Edge/vertex removal}

Let $\h$ be a hypergraph with $e\in E(\h)$. The hypergraph $\h-e$ has vertex set $V(\h)$ and edge set $E(\h)\setminus \{e\}$. 

\begin{thm}\label{prop:edgeremoval}
Let $\h$ be a hypergraph with an edge $e$. Then $\gpi(\h) -1 \leq \gpi(\h - e) \leq \gpi(\h)+|e|-1$. These bounds are tight.
\end{thm}

\begin{proof}
For the lower bound, if we have an infectious power dominating set $\hat{S}$ for $\h-e$ then adding an edge may ruin uniqueness for infection. By adding one vertex of $e$ to $\hat{S}$, this new edge is infected in the domination step. Then $\hat{S}$ will infect the remainder of the graph as it did in $\h-e$. Thus we have $\gpi(\h) \leq \gpi(\h-e)+1$ and so $\gpi(\h)-1 \leq \gpi(\h -e)$. 

For the upper bound, consider an infectious power dominating set $S$ for $\h$. At some point in the infectious power domination process on $\h$, at least one vertex of $e$ must become infected in order for the rest of $e$ to become infected. Call this vertex $v$ and let $e'=e\setminus\{v\}$. When we remove edge $e$, then the  vertices in $e'$ may no longer be infected. Consequently, $S\cup e'$ is an infectious power dominating set of $\h - e$ and has at most $\gpi(\h)+|e|-1$ vertices. 

For tightness of the lower bound, consider $\K_{n_1,\ldots,n_k}^{(k)}$ with $n_1=3$ and $n_i\geq 3$ for all $i\neq 1$. By Proposition \ref{prop:unikpartite}, $\gpi\lp \K_{n_1,\ldots,n_k}^{(k)} \rp =2 $. Remove edge $e=\{v_1,\ldots,v_k\}$ with $v_i\in V_i$ and let $v_1$ have non-neighbors $a$ and $b$. Let $S_0=\{a\}$. In the domination step, every vertex except for $v_1$ and $b$ have been infected. 
Let $A=\{v_2,\ldots,v_k\}$ and $e=A\cup\{b\}$. As $A\cup \{v_1\}$ is not an edge, $A$ infects $b$. Then $v_1$ is the one remaining infected vertex and so becomes infected by any vertex adjacent to $v_1$. Therefore $\gpi(\K_{n_1,\ldots,n_k}^{(k)}-e)=1$.

To see the tightness of the upper bound, consider any connected linear interval hypergraph $\h$ with first edge $e$. By Observation \ref{obs:gpieq1examples} Part 3, $\gpi(\h)=1$. The hypergraph $\h - e$ has $|e|-1$ isolated vertices and the remaining vertices form a connected interval hypergraph. Thus any infectious power dominating set must contain the $|e|-1$ isolated vertices and the vertex that forms an infectious power dominating set for the remainder of the graph. Thus $\gpi(\h - e) = 1 + |e|-1 = \gpi(\h) +|e| -1$.

\end{proof}

For the power domination number, we see that the proofs of the upper and lower bounds in Proposition \ref{prop:edgeremoval} also apply. Thus we have the following.

\begin{cor}
Let $\h$ be a hypergraph with an edge $e$. Then $\gamma_{P}(\h) -1 \leq \gamma_{P}(\h - e) \leq \gamma_{P}(\h)+|e|-1$. 
\end{cor}

We note that removing a vertex and its corresponding edges may drastically change the power domination number. Adding a dominating vertex (i.e., a vertex that is adjacent to every vertex in the graph) will lower the infectious power domination number to 1 regardless of the remaining graph structure. Similarly, removing such a vertex may drastically increase the infectious power domination number. 

\subsection{Linear sums}

The \emph{linear sum} of hypergraphs $\h_1,\h_2$ is $\h_1\star\h_2 = (V(\h_1)\cup V(\h_2), \{e_1\cup e_2 : e_1\in E(\h_1), e_2\in E(\h_2)\})$. Note that 
the linear sum of two 2-hypergraphs (i.e. two graphs) will yield a 4-hypergraph. As the vertex set is the union of the vertex sets of the input hypergraphs, we call this operation a linear sum and use the notation $\h_1\star\h_2$ instead of the term \emph{direct product} and notation $\h_1\times\h_2$ as used in \cite{infection18}.

\begin{thm}\label{thm:direct}
For any connected hypergraphs $\h_1$ and $\h_2$, $\gpi\lp\h_1\star\h_2\rp\leq \gamma_P (\h_1\star\h_2) \leq \gamma (\h_1\star\h_2)\leq 2$. Furthermore, if $\gpi(\h_1)=1$ or $\gpi(\h_2)=1$ then $\gpi(\h_1\star\h_2)=1$.
\end{thm}

\begin{proof}
Take $v_1\in V(\h_1)$ and $v_2\in V(\h_2)$ and let $S_0=\{v_1,v_2\}$.  As each hypergraph is connected, $v_1\in e_1 \in E(\h_1)$ and for any vertex $w_2\in V(\h_2)$, there exists $f_2\in E(\h_2)$ such that $w_2\in f_2$. Then $v_1,w_2 \in e_1\cup f_2 \in E(\h_1\star \h_2)$ and so $w_2$ is adjacent to $v_1$. Similarly, we see that any vertex $w_1\in V(\h_1)$ is adjacent to $v_2$. Therefore every vertex in $\h_1\star \h_2$ is observed in the domination step and so $\gpi(\h_1\star\h_2)\leq \gamma_P (\h_1\star\h_2)\leq \gamma (\h_1\star\h_2) \leq 2$.

Let $\gpi(\h_1)=1$ with $\{v_1\}$ being an infectious power dominating set. We will show that $\{v_1\}$ is also an infectious power dominating set in $\h_1\star\h_2$. In the domination step, all vertices of $\h_2$ are infected as are the original neighbors of $v_1$ from $\h_1$. Suppose that $A\subseteq V(\h_1)$ infected edge $e_1 \in E(\h_1)$ during the infectious power domination process for $\h_1$. This means that for any uninfected vertex $w_1 \in V(\h_1)$, $A\cup \{w_1\} \subseteq e_1'\in E(\h_1)$ implies that $w_1\in e_1$. Let $e_2$ be any edge of $\h_2$. Consider $B=A\cup e_2$. Then for any uninfected $w_1$ in $\h_1\star \h_2$, if $B \cup\{w_1\} \subseteq e'\cup e_2$ then $A\cup \{w_1\} \subseteq e'$ and so $w_1\in e_1$. Thus $w_1\in e_1\cup e_2$ and so $B$ infects $e_1\cup e_2$. Continuing in this way, using the infection steps as in $\h_1$ but with the addition of the vertices in $e_2$,  every vertex of $\h_1$ becomes infected. Therefore, $\gpi(\h_1\star\h_2)=1$.\end{proof}



\begin{cor} 
For any hypergraph $\h_1$ with $\gpi(\h_1)=1$ and any hypergraphs $\h_2,\ldots,\h_\ell$, we have $\gpi\lp \h_1\star\h_2\cdots\star\h_\ell\rp=1$.
\end{cor}


Theorem \ref{thm:direct} shows that the infectious power domination number cannot increase for linear sums. In particular, if $\h_1$ or $\h_2$ has infection power domination number 1, the linear sum must also have infection power domination number 1. This is a stark contrast with the following result for the infection number.

\begin{prop}{\rm \cite[Proposition 6.5]{infection18}}
If $\h_1$ and $\h_2$ are both hypergraphs with more than one edge and $I(\h_1)=I(\h_2)=1$ then $I(\h_1\star\h_2)=2$.
\end{prop}



\subsection{Cartesian products}

For an edge $e=\{w_1,\ldots,w_k\}$ and a vertex $v$ define $e\times v = \{(w_1,v),\ldots,(w_k,v)\}$ and similarly $v\times e = \{(v,w_1),\ldots,(v,w_k)\}$ as in \cite{infection18}. The \emph{Cartesian product} of two hypergraphs $\h_1$ and $\h_2$ is denoted $\h_1\square\h_2$, has vertex set  $ V(\h_1)\times V(\h_2)$, and has edge set 
\[\{e_1\times v_2 : e_1\in E(\h_1), v_2\in V(\h_2)\}\cup\{v_1\times e_2 : v_1\in V(\h_1), e_2\in E(\h_2)\}.\]

The infectious power domination number of the Cartesian product of two graphs may be greater than the infectious power domination number of either graph.  In Proposition \ref{prop:completecartesian} we see that the infectious power domination number of a $k$-complete hypergraph with a $\ell$-complete hypergraph is one such example. Recall from Proposition \ref{prop:kcomplete} that $\gpi(K_n^{(k)})=1$.

\begin{prop}\label{prop:completecartesian} Let $n\leq m$ with $3\leq k \leq n-1$ and $3\leq \ell \leq m-1$. Then
$\gpi\lp\K_n^{(k)} \square \K_m^{(\ell)}\rp = n-1$.
\end{prop}

\begin{proof} Denote the vertices of $\K_n^{(k)}$ by $v_1,\ldots, v_n$ and the vertices of $\K_m^{(\ell)}$ by $w_1,\ldots, w_m$. Then all vertices of $\K_n^{(k)} \square \K_m^{(\ell)}$ are of the form $(v_i, w_j)$. For each $w_r\in V(\K_m^{(\ell)})$, define $K_n^{(k)}\times w_r = \{(v_i,w_r): 1\leq i \leq n\}$. Similarly, for each $v_s\in V(\K_n^{(k)})$ we define $v_s\times K_m^{(\ell)} = \{(v_s,w_j): 1\leq j \leq m\}$. Note that every edge of $\K_n^{(k)} \square \K_m^{(\ell)}$ is a subset of exactly one set of one of these forms.

Consider the set $S_0=\{(v_1,w_1), (v_2,w_2), \ldots, (v_{n-1},w_{n-1}) \}$. In the domination step, as any vertex $(v_i,w_j)$ is adjacent to each vertex in both $\K_n^{(k)}\times w_j$ and $v_i\times\K_m^{(\ell)}$, the only uninfected vertices are $\{(v_n, w_n), (v_n, w_{n+1}), \ldots, (v_n, w_m) \}$.  However, this means that for each $n\leq x \leq m$, the set $A_x=\{(v_{n-\ell},w_x), (v_{n-\ell+1},w_x), \ldots, (v_{n-1},w_x)\}$ is contained in the edge  $e=A_x \cup \{(v_n,w_x)\}$. Moreover, $A$ is not contained in any other edge containing an uninfected vertex as every other vertex in $\K_n^{(k)}\times w_x$ is infected. Hence $A_x$ infects $(v_n,w_x)$ for $n\leq x\leq m$ and so all vertices become infected. Therefore,  $\gpi\lp\K_n^{(k)} \square \K_m^{(\ell)}\rp \leq n-1$.

Assume for eventual contradiction that there exists some infectious power dominating set $S_0$ with $|S_0|\leq n-2$. By the Pigeonhole Principle, since $m\geq n \geq 4$, there are at least two $i$ so that $v_i\times \K_m^{(\ell)}$ contains no vertex of $S_0$. Without loss of generality, let these be $i$ and $i'$. In the same way, there exists $j$ and $j'$ so that $\K_n^{(k)}\times w_j$ and $\K_n^{(k)}\times w_{j'}$ contain no vertex of $S_0$. Now consider the vertices $(v_i,w_j), (v_i, w_{j'}), (v_{i'},w_{j})$, and $(v_{i'},w_{j'})$.

We show that there is no set that can infect these vertices. Any set which could infect $(v_i,w_j)$ must be of the form $A \subseteq \K_n^{(k)}\times w_j \setminus \{(v_i,w_j), (v_{i'},w_j)\}$ with $|A|\leq k-1$. We see that by construction of the $k$-complete hypergraph, there is a edge $e$ of $\K_n^{(k)}\times w_j$ containing $A \cup \{(v_i,w_j)\}$. There is also an edge $e'$ so that $A\cup \{(v_{i'},w_j)\} \subseteq \lp e\setminus \{(v_{i},w_j)\} \rp \cup \{(v_{i'},w_j)\} = e'$. Thus any such set $A$ of infected vertices that are in an edge with $(v_i,w_j)$ are also in a different edge with $(v_{i'},w_j)$ and so no such $A$ can infect  $(v_i,w_j)$ or $(v_{i'},w_j)$. In a similar way, we can consider $A'\subseteq \K_n^{(k)}\times w_{j'} \setminus \{(v_i,w_{j'}), (v_{i'},w_{j'})\}$ with $|A'|\leq k-1$, $A''\subseteq v_i\times \K_m^{(\ell)} \setminus \{(v_i,w_j), (v_i,w_{j'})\}$ with $|A''|\leq \ell-1$, and $A'''\subseteq v_{i'}\times \K_m^{(\ell)} \setminus \{(v_{i'},w_j), (v_{i'},w_{j'})\}$ with $|A'''|\leq \ell-1$ to see that no such subsets can infect any of $(v_i,w_j), (v_i, w_{j'}), (v_{i'},w_{j})$, or $(v_{i'},w_{j'})$. Therefore, there is no possible set of infected vertices that can infect these four vertices when $|S_0|\leq n-2$. Hence $\gpi\lp\K_n^{(k)} \square \K_m^{(\ell)}\rp \geq n-1$.

\end{proof}

Bergen et al. \cite{infection18} established a general upper bound on the infection number of Cartesian products which is greater than the infection number of either hypergraph.  
\begin{prop}{\rm\cite[Corollary 6.11]{infection18}}.
Let $\h_1$ and $\h_2$ be hypergraphs, then $I(\h_1\square \h_2) \leq I(\h_1)\lp I(\h_2) + |E(\h_2)| \rp$.
\end{prop}
\noi 

\subsection{Weak coronas}

The \emph{weak corona} of a $k$-uniform hypergraph $\G^{(k)}$ with a $(k-1)$-hypergraph $\h^{(k-1)}$, as defined in \cite{infection18}, is the $k$-uniform hypergraph denoted by $\G^{(k)}\circ_{w} \h^{(k-1)}$, with vertex set 
\[V\lp\G^{(k)}\circ_{w} \h^{(k-1)}\rp=\ds V(\G^{(k)}) \cup \lp\bigcup_{v\in V(\G^{(k)})} V\lp \h_v^{(k-1)} \rp \rp\] 
where $\h_v^{(k-1)}$ is a copy of $\h^{(k-1)}$ corresponding to vertex $v\in V\lp\G^{(k)}\rp$, and edge set
\[E\lp\G^{(k)}\circ_{w} \h^{(k-1)}\rp = E\lp\G^{(k)} \rp \cup \left\{ e_v \cup \{v\} : v\in V\lp \G^{(k)}\rp,  e_v\in E\lp \h_v^{(k-1)} \rp \right\}.\]
That is, the weak corona is formed by taking a copy of $\h^{(k-1)}$ for each vertex $v$ of $\G^{(k-1)}$ and then adding $v$ to each edge of its copy of $\h^{(k-1)}$. 

In \cite{infection18}, it was shown that $I\lp \G^{(k)}\circ_{w} \h^{(k-1)} \rp \leq |V(\G^{(k)})| I\lp \h^{(k-1)}\rp$. We obtain a better result, reflecting the strength of the domination step.

\begin{prop}
For any hypergraphs $\G^{(k)}$ and $\h^{(k-1)}$, 
\[\gpi\lp \G^{(k)}\circ_{w} \h^{(k-1)} \rp \leq \gamma_{P}\lp \G^{(k)}\circ_{w} \h^{(k-1)} \rp \leq  \gamma\lp \G^{(k)}\circ_{w} \h^{(k-1)} \rp \leq |V\lp\G^{(k)}\rp|.\] 
This bound is tight whenever $\h^{(k-1)}$ has at least two edges.
\end{prop}

\begin{proof}
Observe that $V(\G^{(k)})$ is a dominating set. 

For tightness, let $\G^{(k)}$ be any hypergraph and let $\h^{(k-1)}$ be any $(k-1)$-hypergraph with at least two edges. 
Suppose for contradiction that $\gpi\lp\G^{(k)}\circ_w \h^{(k-1)}\rp < |V(\G^{(k)})|$ and let $S_0$ be a minimum infectious power dominating set. Then there exists some $v\in V(\G^{(k)})$ so that $V(\h^{(k-1)}_v)\cup \{v\}$ contains no vertex of $S_0$. By the construction of the weak corona, the only vertex of $V(\h^{(k-1)}_v)\cup \{v\}$ adjacent to a vertex outside of $V(\h^{(k-1)}_v)\cup \{v\}$ is $v$.  Thus the first vertex of $V(\h^{(k-1)}_v)\cup \{v\}$ to become infected is $v$. As $\h^{(k-1)}$ has at least two edges, we have $e,e'\in E(\h_v^{(k-1)})$. Then $v$ is adjacent to uninfected vertices in the edges $e\cup\{v\}$ and $e'\cup\{v\}$. 
However, this implies that $V(\h_v^{(k-1)})$ cannot become infected, a contradiction. Thus $\gpi \lp \G^{(k)}\circ_{w} \h^{(k-1)} \rp \geq |V(\G^{(k)})|$.

\end{proof}

Let $P_3^{(2)}$ denote the path graph on 3 vertices. 

\begin{cor}\label{cor:n4weakcorona}
Let $\G^{(3)}$ be any 3-uniform hypergraph. Then \[\gpi\lp \G^{(3)} \circ_{w} P_3 \rp = \gamma_{P}\lp \G^{(3)} \circ_{w} P_3^{(2)} \rp = |V\lp\G^{(3)}\rp|.\]
\end{cor}

Corollary \ref{cor:n4weakcorona} demonstrates an infinite family of $3$-uniform hypergraphs that achieves the bound in Conjecture \ref{conj:noverfour} for both the power domination number and the infectious power domination number. Furthermore, $V(G)$ is an infectious power dominating set for which each vertex has exactly 3 private neighbors.

\section{Concluding remarks}\label{sec:conclusion}


Several interesting questions remain for both infectious power domination and power domination for hypergraphs. 

We have not yet found an generalized improvement of the bound in Corollary \ref{cor:nover3bound} or a counterexample to the bound in Conjecture \ref{conj:noverfour} and so this remains an open question for further study. We have seen two examples of infinite families that achieve the conjectured bound, in Proposition \ref{prop:n4lollipops} and Corollary \ref{cor:n4weakcorona}. Each of these families achieves the bound with a dominating set. However, the family $\lolli$ indicates that a counting argument similar to that used in \cite{zhaokangchang06} will not work for hypergraphs. 

While we know that the spider cover number is only an upper bound for the infectious power domination number for hypertrees as seen in Example \ref{ex:hypertree}, we do not know if there is a useful \emph{lower} bound.

Cartesian products also present a large number of open questions. We showed in Proposition \ref{prop:completecartesian} that the infectious power domination number of the Cartesian product of two hypergraphs can be greater than either hypergraph but we do not have a general upper or lower bound.

\section*{Acknowledgments}
This research was supported by the US Department of Defense's Science, Mathematics and Research for Transformation (SMART) Scholarship for Service Program.

\bibliographystyle{plain}
\bibliography{hypergraphs}

\begin{thebibliography}{10}

\bibitem{AIMminrank}
{AIM Minimum Rank -- Special Graphs Work Group}.
\newblock Zero forcing sets and the minimum rank of graphs.
\newblock {\em Linear Algebra and its Applications}, 428(7):1628 -- 1648, 2008.

\bibitem{bfffhvw18}
K.F. Benson, D~Ferrero, M~Flagg, V~Furst, L~Hogben, V~Vasilevska, and
  B~Wissman.
\newblock Zero forcing and power domination for graph products.
\newblock {\em Australasian Journal of Combinatorics}, 70:221--235, 01 2018.

\bibitem{infection18}
Ryan Bergen, Shaun~M. Fallat, Adam Gorr, Ferdinand Ihringer, Karen Meagher,
  Alison Purdy, Boting Yang, and Guanglong Yu.
\newblock Infection in hypergraphs.
\newblock {\em Discrete Applied Mathematics}, 237:43--56, 2018.

\bibitem{bretto}
Alain Bretto.
\newblock {\em Hypergraph Theory: An Introduction}.
\newblock Springer Publishing Company, Incorporated, 2013.

\bibitem{brueniheath05}
Dennis~J. Brueni and Lenwood~S. Heath.
\newblock The {PMU} placement problem.
\newblock {\em {SIAM} J. Discrete Math.}, 19(3):744--761, 2005.

\bibitem{quantum}
Daniel Burgarth and Vittorio Giovannetti.
\newblock Full control by locally induced relaxation.
\newblock {\em Physical review letters}, 99 10:100501, 2007.

\bibitem{kpower15}
Gerard~Jennhwa Chang and Nicolas Roussel.
\newblock On the \emph{k}-power domination of hypergraphs.
\newblock {\em J. Comb. Optim.}, 30(4):1095--1106, 2015.

\bibitem{histpowdom}
Teresa~W. Haynes, Sandra~Mitchell Hedetniemi, Stephen~T. Hedetniemi, and
  Michael~A. Henning.
\newblock Domination in graphs applied to electric power networks.
\newblock {\em {SIAM} J. Discrete Math.}, 15(4):519--529, 2002.

\bibitem{henninglowenstein12}
Michael~A. Henning and Christian L{\"{o}}wenstein.
\newblock Hypergraphs with large domination number and with edge sizes at least
  three.
\newblock {\em Discrete Applied Mathematics}, 160(12):1757--1765, 2012.

\bibitem{hoghyp18}
Leslie Hogben.
\newblock Zero forcing and maximum nullity for hypergraphs.
\newblock \url{https://arxiv.org/abs/1808.09908}, 2018.

\bibitem{zhaokangchang06}
Min Zhao, Liying Kang, and Gerard~J. Chang.
\newblock Power domination in graphs.
\newblock {\em Discrete Mathematics}, 306(15):1812--1816, 2006.

\end{thebibliography}

\end{document}